\documentclass{amsart}%
\usepackage{amsfonts}
\usepackage{amsmath}
\usepackage{amssymb}
\usepackage{graphicx}%
\providecommand{\U}[1]{\protect\rule{.1in}{.1in}}
\newtheorem{theorem}{Theorem}
\theoremstyle{plain}

\newtheorem{conjecture}{Conjecture}
\newtheorem{corollary}{Corollary}

\newtheorem{lemma}{Lemma}

\newtheorem{proposition}{Proposition}
\newtheorem{remark}{Remark}

\numberwithin{equation}{section}
\begin{document}
\title[Moment sequences]{Beta distributions whose moment sequences are related to integer sequences
listed in OEIS}
\author{Pawe\l \ J. Szab\l owski}
\address{Department of Mathematics and Information Sciences, \\
Warsaw University of Technology\\
ul Koszykowa 75, 00-662 Warsaw, Poland}
\email{pawel.szablowski@gmail.com}
\thanks{The author is grateful to the referees for their remarks and suggestions. The
first one inspired the author to formulate and prove Theorem \ref{integer}.
The second one pointed out numerous inaccuracies, misspellings and
inconsistencies in the previous version of the paper.}
\date{December, 2021}

\begin{abstract}
We recall some basic properties of the Beta distribution and some of its
modifications. We identified around $20$ of the moment sequences of Beta
distributions as important integer sequences in the OEIS base of integer
sequences. Among those identified are Catalan, Riordan, Motzkin, or 'super
ballot numbers'. By applying a method of expansion of the ratio of densities
of involved distributions we are able to obtain some known and many unknown
relationships between e.g. Catalan numbers and other moment sequences of the
Beta distributions.

\end{abstract}
\maketitle

\section{Introduction}

This paper intends to show that many well-known discrete sequences of numbers
that are important in many, distant from the theory of probability, branches
of mathematics are, in fact, closely connected with moment sequences of the
beta distribution. By close connection we mean some simple operation like, for
example, multiplication by a sequence of powers of some number or binomial transformation.

The idea of representing known integer sequences as moment sequence is
becoming more and more popular in recent years. Nice arguments to follow this
idea were recently presented in the paper of Sokal \cite{Sokal20}. There is
also a nice review of basic facts concerning moment sequence as well as some
criteria both sufficient and necessary and only sufficient for a number
sequence to be a moment sequence.

Some of these arguments and facts we repeat here below, however, to get the
definition of the moment sequences and their basic properties as well as the
properties of the set of such sequences we refer the reader to the paper of
Sokal or to the appendix of the recently published, paper of Szab\l owski
\cite{Szabl21}.

As far as moments of the Beta distribution are concerned, it turns out that,
for example, sequences of Catalan, Motzkin, Riordan or 'super ballot' numbers
are such sequences. They are, in fact, moments of some modification the
classical beta considered on the segment $[0,1]$. In order to indicate briefly
what modifications we mean firstly we will recall the definition of Beta
distribution and the modifications that we are going to consider.

We will be dealing with the moment sequences $\left\{  m_{n}\right\}  $ i.e.
sequences that are defined by the following formula%
\[
m_{n}\allowbreak=\allowbreak\int x^{n}d\mu\left(  x\right)  ,
\]
$n\geq0$, where $\mu$ denotes a positive measure on the real line.

Recall, that a numerical sequence $\left\{  m_{n}\right\}  _{n\geq0}$ with
$m_{0}\allowbreak=\allowbreak1,$ is a moment sequence of a probability
distribution with infinite support iff all its Hankel matrices $H_{n}:=\left[
m_{i+j}\right]  _{0\leq i+j\leq n}$ are positive definite, or equivalently,
the values if all determinants $\left\{  \det[H_{n}]\right\}  _{n\geq0}$ have
positive values. Additionally, if the distribution whose moments are elements
of the sequence $\left\{  m_{n}\right\}  $ has the support contained in the
non-negative axis, then the moment sequence satisfies also the following
condition: The following sequence $\left\{  \det H_{n}^{^{\prime}}\right\}
_{n\geq0}$ assumes also non-negative values. Here matrix $H_{n}^{^{\prime}}$
has $\left\{  i,j\right\}  -$th entry equal to $m_{1+i+j}$, for $0\leq i,j\leq
n$.

Let us recall the general property of the moment sequences that will be of use
in the sequel.

\begin{proposition}
\label{momenty}Suppose $\left\{  a_{n}\right\}  _{n\geq0}$ and $\left\{
b_{n}\right\}  _{n\geq0}$ are two moment sequences. Then $\left\{  a_{n}%
b_{n}\right\}  _{n\geq0}$, $\left\{  \sum_{j=0}^{n}\binom{n}{j}a_{j}%
b_{n-j}\right\}  _{n\geq0}$ and $\left\{  \sum_{j=0}^{n}\left(  -1\right)
^{j}\binom{n}{j}a_{j}b_{n-j}\right\}  _{n\geq0}$ are also moment sequences.
\end{proposition}

\begin{proof}
For the proof see ,e.g., either \cite{Benn11} or \cite{Szab23}.
\end{proof}

Finally, by expanding the ratio of the densities, whose moments we are
considering, and then integrating, we get for free the relationships between
involved moments. The idea of expanding the ratio of the densities (the
Radon-Nikodym derivative more generally) has been presented in
\cite{Szablowski2010(1)} and later developed and generalized in
\cite{SzabChol}. It happens very often that the sequence of the moments has a
deep combinatorial interpretation and by using this method of expansion, we
can get relationships between these important combinatorial sequences that are
usually difficult to prove by combinatorial means. For example, we can easily
show that
\begin{align*}
\binom{2n}{n}  &  =4^{n}-2\sum_{j=0}^{n-1}C_{j}4^{n-1-j},\\
C_{n}  &  =\frac{3}{2}\sum_{i\geq0}\frac{1}{4^{i}}\frac{(2n+2i)!}%
{(i+n)!(n+i+2)!},
\end{align*}
where $\left\{  C_{n}\right\}  _{n\geq0}$ are the so-called Catalan numbers.
Numbers $\left\{  C_{n}\right\}  $ constitute, as it turns out, the moment
sequence of the distribution with the following density:
\[
\frac{1}{2\pi}\sqrt{\frac{4-x}{x}},
\]
where $0<x<4$.

We will go into detail and give more examples in the sequel. The paper is
organized as follows. In the next Section \ref{bas}, we present a definition
of Beta distribution, its modifications and the basic properties of its
moments. The following Section \ref{first}, is dedicated to the presentation
of particular examples, mostly concerning cases when $\alpha$ and $\beta$ are
multiples of $1/2$. In this section, we identify inside the basis of integer
sequences OEIS, many moments of the beta distribution. The last Section,
\ref{expan}, is devoted to the presentation of some (by no means all) possible
expansions finite as well as infinite expansions of elements of one sequence
in terms of the other.

\section{Basic ingredients and properties\label{bas}}

Let us recall the so-called beta distribution, i.e., the distribution with the
density defined for $\left\vert x\right\vert <1$ and $\alpha,\beta>0,$
\begin{equation}
a(x;\alpha,\beta)=\frac{x^{\alpha-1}(1-x)^{\beta-1}}{B(\alpha,\beta)},
\label{Jac}%
\end{equation}
where $B(\alpha,\beta)$ denotes the value of the well-known Beta function
taken at $\alpha$ and $\beta.$ It is well-known that the sequence of moments
of this distribution is given by the formula:%
\begin{equation}
M_{n}(0,\alpha,\beta)/4^{n}=\int_{0}^{1}x^{n}a(x;\alpha,\beta)dx=\frac
{\alpha^{(n)}}{(\alpha+\beta)^{(n)}}. \label{mJ}%
\end{equation}
We use here redundant, at first-sight, notation. The reason for this will be
clear in the sequel.

In formula (\ref{mJ}), we used the following notation. For $x\in\mathbb{C}$
let us denote:%
\begin{equation}
(x)_{(n)}\allowbreak=\allowbreak x(x-1)\ldots(x-n+1). \label{down}%
\end{equation}
This polynomial in $x$ will be called falling factorial while the following
polynomial
\begin{equation}
(x)^{(n)}=x(x+1)\ldots(x+n-1), \label{up}%
\end{equation}
will be called raising factorial. In both cases we set $1$ when $n\allowbreak
=\allowbreak0$.

It is also well-know that
\begin{equation}
(x)_{(n)}=(-1)^{n}(-x)^{(n)},\text{ and }(x)^{(n)}=(-1)^{n}(-x)_{(n)}.
\label{Pochh}%
\end{equation}
Recall, that we have also the so-called binomial theorem stating that for all
complex $\left\vert x\right\vert <1$ we have:
\begin{equation}
(1-x)^{\alpha}\allowbreak=\allowbreak\sum_{j\geq0}(-x)^{j}(\alpha
)_{(j)}/j!=\sum_{j\geq0}x^{j}(-\alpha)^{(j)}/j!. \label{bin}%
\end{equation}
The modifications of the beta distribution that we are going to consider, are
the following. Namely, we will examine the distribution family as well as
densities.
\begin{equation}
g(x;c,\alpha,\beta)=\frac{(x-c)^{\alpha-1}(4+c-x)^{\beta-1}}{4^{\alpha
+\beta-1}B(\alpha,\beta)}, \label{dens}%
\end{equation}
supported on the segment $[c,4+c],$ $c\in\mathbb{R}$. If we denote by $X$ the
random variable with the density $a(s;\alpha,\beta)$ and by $Y$ the random
variable with the density $g(x;c,\alpha,\beta)$, then we see that
\[
Y=4X+c.
\]
Let us denote by $M_{n}(c,\alpha,\beta)$ the $n-$th moment of $Y$, i.e.,
\begin{equation}
M_{n}(c,\alpha,\beta)=\int_{c}^{4+c}x^{n}g(x;c,\alpha,\beta)dx. \label{ogM}%
\end{equation}

Note that for all $\alpha,\beta>0$ and $c\in\mathbb{R}:$ $M_{0}(c,\alpha
,\beta)\allowbreak=\allowbreak1$.

We have the following Lemma:

\begin{lemma}
\label{binomial}i) $\forall b,c\in\mathbb{R}$, $\alpha,\beta>0$ and
$n\in\mathbb{N\cup}\left\{  0\right\}  $ we have:%
\begin{equation}
M_{n}(b,\alpha,\beta)\allowbreak=\allowbreak\sum_{j=0}^{n}\binom{n}{j}%
M_{j}(c,\alpha,\beta)(b-c)^{n-j}. \label{MbMc}%
\end{equation}
In particular, we get:
\begin{align}
M_{n}(c,\alpha,\beta)  &  =\sum_{j=0}^{n}\binom{n}{j}M_{j}(0,\alpha
,\beta)c^{n-j},\label{Mm}\\
4^{n}\frac{\alpha^{(n)}}{(\alpha+\beta)^{(n)}}  &  =M_{n}(0,\alpha,\beta
)=\sum_{j=0}^{n}\binom{n}{j}M_{j}(c,\alpha,\beta)(-c)^{n-j}. \label{iMm}%
\end{align}

ii)
\begin{align*}
M_{n}(c,\alpha+1,\beta)  &  =\frac{\alpha+\beta}{4\alpha}(M_{n+1}%
(c,\alpha,\beta)-cM_{n}(c,\alpha,\beta),\\
M_{n}(c,\alpha,\beta+1)  &  =\frac{\alpha+\beta}{4\beta}((4+c)M_{n}%
(c,\alpha,\beta)-M_{n+1}(c,\alpha,\beta)).
\end{align*}

iii) If $c\allowbreak=\allowbreak0$ we have also for $n\geq1$
\begin{align}
M_{n}(0,\alpha,\beta)  &  =\frac{4\alpha}{\alpha+\beta}M_{n-1}(0,\alpha
+1,\beta),\label{iiia}\\
M_{n}(0,\alpha,\beta)  &  =\frac{\beta}{\alpha+\beta}\sum_{j\geq0}%
M_{n+j}(0,\alpha,\beta+1). \label{iiib}%
\end{align}

\end{lemma}

\begin{proof}
i) Let us denote by $Y_{b}\allowbreak=\allowbreak4X\allowbreak+\allowbreak
b\allowbreak=\allowbreak4X\allowbreak+\allowbreak c\allowbreak+\allowbreak
(b-c)\allowbreak=\allowbreak Y_{c}\allowbreak+\allowbreak(b-c),$ so we have:
\begin{align*}
M_{n}(b,\alpha,\beta)  &  =EY_{b}^{n}\allowbreak=\sum_{j=0}^{n}\binom{n}%
{j}EY_{b}^{j}(b-c)^{n-j}\\
&  =\sum_{j=0}^{n}\binom{n}{j}M_{j}(c,\alpha,\beta)(b-c)^{n-j}.
\end{align*}
ii) We have:
\begin{gather*}
M_{n}(c,\alpha+1,\beta)=\frac{\alpha+\beta}{4\alpha}\int_{c}^{4+c}%
(x-c)x^{n}g(x;c,\alpha,\beta)dx\\
=\frac{\alpha+\beta}{4\alpha}(M_{n+1}(c,\alpha,\beta)-cM_{n}(c,\alpha
,\beta),\\
M_{n}(c,\alpha,\beta+1)=\frac{\alpha+\beta}{4\beta}\int_{c}^{4+c}%
(4+c-x)x^{n}g(x;c,\alpha,\beta)dx\\
=\frac{\alpha+\beta}{4\beta}((4+c)M_{n}(c,\alpha,\beta)-M_{n+1}(c,\alpha
,\beta)).
\end{gather*}

iii) First of all notice that $B(\alpha+1,\beta)/B\left(  \alpha,\beta\right)
\allowbreak=\allowbreak\frac{\alpha}{\alpha+\beta}$, $g(x;0,\alpha
,\beta)/g(x;0,\alpha+1,\beta)\allowbreak=\allowbreak\frac{4\alpha}%
{(\alpha+\beta)x}.$ Now,
\[
M_{n}(0,\alpha,\beta)\allowbreak=\allowbreak\frac{4\alpha}{\alpha+\beta
}M_{n-1}(0,\alpha+1,\beta).
\]
Similarly
\[
\frac{g(x;0,\alpha,\beta)}{g(x;0,\alpha,\beta+1)}\allowbreak=\allowbreak
\frac{4\beta}{(\alpha+\beta)(4-x)}\allowbreak=\allowbreak\frac{\beta}%
{\alpha+\beta}\sum_{j\geq0}(x/4)^{j},
\]
and notice that the series is convergent and increasing for $x\in\lbrack0,4).$
Thus, by the Lebesgue theorem on the monotone convergence we get the second assertion.
\end{proof}

\begin{remark}
Notice also, that we have
\begin{equation}
\frac{M_{n}(b,\alpha,\beta)}{(b-c)^{n}}\allowbreak=\allowbreak\sum_{j=0}%
^{n}\binom{n}{j}(-1)^{j}\frac{M_{j}(c,\alpha,\beta)}{(c-b)^{j}}. \label{bt}%
\end{equation}
Thus, the sequence $\left\{  M_{n}(b,\alpha,\beta)/(b-c)^{n}\right\}
_{n\geq0}$ is the so-called binomial transformation of the sequence $\left\{
M_{n}(c,\alpha,\beta)/(c-b)^{n}\right\}  _{n\geq0}$. Conversely, notice that
the sequence $\left\{  M_{n}(c,\alpha,\beta)/(c-b)^{n}\right\}  _{n\geq0}$ is
the binomial transform of the sequence \newline$\left\{  M_{n}(b,\alpha
,\beta)/(b-c)^{n}\right\}  _{n\geq0}$.

However, in the very important base of integer sequences OEIS there is
slightly different definition of the binomial transform.

Namely, given two sequences $\left\{  a_{n}\right\}  _{n\geq0}$ and $\left\{
b_{n}\right\}  _{n},$ if we have for all $n\geq0$:
\[
b_{n}=\sum_{j=0}^{n}\binom{n}{j}a_{j},
\]
then, we say, that sequence $\left\{  b_{n}\right\}  $ is the binomial
transform of the sequence $\left\{  a_{n}\right\}  .$ Note, that then we also
have for all $n\geq0$:
\[
a_{n}=\sum_{j=0}^{n}(-1)^{j}\binom{n}{j}b_{j}.
\]
Then, according to the OEIS terminology sequence $\left\{  a_{n}\right\}  $ is
the inverse binomial transform of the sequence $\left\{  b_{n}\right\}  .$

Hence, using this terminology, the sequence $\left\{  M_{n}(b,\alpha
,\beta)/(b-c)^{n}\allowbreak\right\}  ,$ is the binomial transform of the
sequence $\left\{  M_{n}(c,\alpha,\beta)/(b-c)^{n}\right\}  .$

Furthermore, it follows from these considerations that to get moments
$M_{n}(c,\alpha,\beta)$ it is enough to calculate moments $M_{n}%
(0,\alpha,\beta)$ and then apply the appropriate binomial transformation.
\end{remark}

Since many integer sequences in OEIS are identified by their generating
functions, we will calculate also generating functions of many of these
integer sequences. Let us remark, that we need to have a generating function
defined only on the small open interval around $0$. What matters, are the
coefficients of its Taylor expansion around zero. That is why all considered
below generating functions will be considered for $x\in(-\delta,\delta),$
$\delta>0.$ Very often $\delta$ will be equal to $1$.

\begin{lemma}
\label{gen}Let us consider the sequence of moments given by the formula
(\ref{mJ}) such that $\alpha+\beta$ is an integer, then

i)%
\begin{gather*}
g(t;\alpha,\beta)=\sum_{j\geq0}t^{j}M_{j}(0,\alpha,\beta)/4^{j}\\
=\left\{
\begin{array}
[c]{ccc}%
1/(1-t)^{\alpha} & \text{if} & \alpha+\beta=1\\
(1-(1-t)^{1-\alpha})/(t(1-\alpha)) & \text{if} & \alpha+\beta=2\\
((1-t)^{2-\alpha}+t(2-\alpha)-1)/(t^{2}(1-\alpha)(2-\alpha)) & \text{if} &
\alpha+\beta=3\\
\frac{(6(1-x)^{3-\alpha}-(\alpha^{2}x^{2}-5\alpha x^{2}+6x^{2}+2\alpha
x-6x+2))}{x^{3}(\alpha-1)(\alpha-2)(\alpha-3)} & \text{if} & \alpha+\beta=4
\end{array}
\right.  ,
\end{gather*}
for $t\in(-1,1)$.

ii) If $g(x)$ is a generating function of $\{M_{n}(0,\alpha,\beta)/4^{n}\}$
then $\frac{(\alpha+\beta)}{x\alpha}(g(x)-1)$ is a generating function of
$\{M_{n}(0,\alpha+1,\beta)/4^{n}\}$, while $\frac{(\alpha+\beta)}{4x\beta
}(1-(1-4x)g(x))$ is a generating function of $\{M_{n}(0,\alpha,\beta
+1)/4^{n}\}$.

iii) Let $g(x)$ be a generating function of the sequence $\left\{
f_{n}\right\}  _{n\geq0}$ i.e., $g(x)\allowbreak=\allowbreak\sum_{n\geq0}%
f_{n}x^{n},$ then $\frac{1}{1-cx}g(\frac{x}{1-cx})\allowbreak$ is the
generating function of the sequence \newline$\left\{  \sum_{j=0}^{n}\binom
{n}{j}f_{j}c^{n-j}\right\}  _{n\geq0}$.
\end{lemma}

\begin{proof}
i) Recall, that $(k)^{(n)}\allowbreak=\allowbreak(n+k-1)!/(k-1)!$ for all
nonnegative integers $k,n$ such that $k+n\neq0$. We get for $\alpha
+\beta\allowbreak=\allowbreak1:$%
\[
\sum_{j\geq0}t^{j}M_{j}(0,\alpha,\beta)/4^{j}=\sum_{j=0}^{\infty}t^{j}%
(\alpha)^{(j)}/j!=\sum_{j=0}^{\infty}t^{j}(-1)^{j}(-\alpha)_{(j)}/j!=\frac
{1}{(1-t)^{a}},
\]
by the binomial theorem (\ref{bin}). When $\alpha+\beta\allowbreak
=\allowbreak2$, we have:%
\begin{gather*}
\sum_{j\geq0}t^{j}M_{j}(0,\alpha,\beta)/4^{j}=\sum_{j=0}^{\infty}t^{j}%
(\alpha)^{(j)}/(j+1)!\\
=\frac{-1}{t(1-\alpha)}\sum_{j=0}^{\infty}(\alpha-1)t^{j+1}(\alpha
)^{(j)}/(j+1)!\\
=\frac{-1}{t(1-\alpha)}\sum_{j=0}^{\infty}t^{j+1}(\alpha-1)^{(j+1)}/(j+1).
\end{gather*}
We act likewise when $\alpha+\beta\allowbreak=\allowbreak3$ and when
$\alpha+\beta\allowbreak=\allowbreak4$.

ii) These assertions are based on the following observations.
\[
\frac{a(x;\alpha+1,\beta)}{a(x;\alpha,\beta)}\allowbreak=\allowbreak
x\frac{\alpha+\beta}{\alpha},
\]
since $\Gamma(\alpha+1)\allowbreak=\allowbreak\alpha\Gamma(\alpha)$, where
$a(x;\alpha,\beta)$ is given by (\ref{Jac}). Similarly%
\[
a(x;\alpha,\beta+1)/a(x;\alpha,\beta)\allowbreak=\allowbreak(1-x)\frac
{\alpha+\beta}{\beta}.
\]
Thus, we have $M_{n}(0,\alpha+1,\beta)\allowbreak=\allowbreak\frac
{\alpha+\beta}{4\alpha}M_{n+1}(0,\alpha,\beta)$ and $M_{n}(0,\alpha
,\beta+1)\allowbreak$\newline$=\allowbreak\frac{\alpha+\beta}{\beta}%
(M_{n}(0,\alpha,\beta)-M_{n+1}(0,\alpha,\beta)/4)$.

iii) After applying ordinary change of the order of summation, we have:
\begin{gather*}
\sum_{n\geq0}^{\infty}x^{n}\sum_{j=0}^{n}\binom{n}{j}f_{j}c^{n-j}=\sum
_{j\geq0}x^{j}f_{j}\sum_{n\geq j}(cx)^{n-j}\frac{(j+1)\ldots n}{(n-j)!}\\
=\sum_{j\geq0}x^{j}f_{j}\sum_{n\geq j}(cx)^{n-j}\frac{(j+1)\ldots n}%
{(n-j)!}=\sum_{j\geq0}x^{j}f_{j}\sum_{k\geq0}(cx)^{k}\frac{(j+1)^{(k)}}{k!}.
\end{gather*}
Now we recall (\ref{Pochh}) and (\ref{bin}) and get:%
\begin{align*}
\sum_{n\geq0}^{\infty}x^{n}\sum_{j=0}^{n}\binom{n}{j}f_{j}c^{n-j}  &
=\sum_{j\geq0}x^{j}f_{j}\sum_{k\geq0}(cx)^{k}\frac{(-1)^{k}(-(j+1))_{(k)}}%
{k!}\\
&  =\sum_{j\geq0}x^{j}c_{j}(1-cx)^{-j-1}=\frac{1}{1-cx}g(\frac{x}{1-cx}).
\end{align*}

\end{proof}

One of the referees posed the following question. For what triplets
$(c,\alpha,\beta)$ the sequence $\left\{  M_{n}(c,\alpha,\beta)\right\}
_{n\geq0}$ generates integers. Well, we will not answer this question fully,
since it seems that a problem is more difficult than expected. We have,
however the following partial result, exposing the possible complications in
answering this question. Namely, we have:

\begin{theorem}
\label{integer}Let $\alpha\in(0,1)$ be a rational number and let
$\alpha\allowbreak=\allowbreak\frac{p}{r}$, where $p$ and $r$ are two positive
integers relatively prime and let $n$ be a positive integers. Then $\forall
n\in\mathbb{N}$ the number
\begin{equation}
\left(  \frac{r}{4}\right)  ^{n}\prod_{j=1}^{k}d_{j}^{\sum_{m=1}^{\infty
}\left\lfloor n/d_{j}^{m}\right\rfloor }M_{n}(0,\frac{p}{r},1-\frac{p}%
{r})=\frac{\prod_{j=0}^{n}(jr+p)}{n!}\prod_{j=1}^{k}d_{j}^{\sum_{m=1}^{\infty
}\left\lfloor n/d_{j}^{m}\right\rfloor }, \label{beta_int}%
\end{equation}
an an integer. Here $r\allowbreak=\allowbreak\prod_{j=1}^{k}d_{j}^{\beta_{j}}%
$, is a prime decomposition of $r$.
\end{theorem}

\begin{proof}
Let us fix $n.$ Then we have by (\ref{mJ}):
\begin{equation}
r^{n}M_{n}(0,\frac{p}{r},1-\frac{p}{r})/4^{n}=r^{n}\frac{(\alpha)^{(n)}}{n!}.
\label{ro}%
\end{equation}
Now let us recall the so-called Chinese reminder Theorem stating that every
two congruence equations
\[
ax\equiv c,(\operatorname{mod}m_{1})~~bx\equiv d~~(\operatorname{mod}m_{2}),
\]
have unique solution $\operatorname{mod}m_{1}m_{2}$ if and only if numbers
$m_{1}$ and $m_{2}$ are relatively prime. Hence, taking $m_{1}\allowbreak
=\allowbreak r$ and $1<n_{1}\leq n$ relatively prime, we see that the set of
two congruence equations:%
\[
x\equiv p,~~\operatorname{mod}r~~x\equiv0,~~\operatorname{mod}n_{1}%
\]
has a unique solution $\operatorname{mod}(n_{1}r).$ In other words, that among
numbers $jr+p$, $j\allowbreak=\allowbreak1,\ldots,n$ at least one is divided
by $n_{1}.$ That means that the number (\ref{ro}) is a rational number that
has the denominator composed by the numbers that are not relatively prime to
$r$. Or in other words, are composed of powers of the divisors of $r.$ Now,
one has to count sums of powers of particular divisors less or equal then $n$.
Let us fix the divisor, let it be $d_{j}.$ Then there are $\left\lfloor
n/d_{j}\right\rfloor $ factors of the form $kd_{j},$ $k\allowbreak
=\allowbreak1,\ldots,$ then there are $\left\lfloor n/d_{j}^{2}\right\rfloor $
factors of the form $kd_{j}^{2}$ and so on. Now since we multiply those
factors (we had $n!$ before canceling out factor relatively prime to $r$) so
we have $d_{j}^{\left\lfloor n/d_{j}\right\rfloor +\left\lfloor n/d_{j}%
^{2}\right\rfloor +\ldots}$ in the denominator of (\ref{ro}). Now to get an
integer out of this number we have to multiply it by the denominator.
\end{proof}

\begin{remark}
If one considers the more general situation like e.g., $M_{n}%
(0,p/r,k-p/r)/4^{n}$ then the situation is not that simple. Namely, in the
denominator, there might appear prime factors that are different than the
prime factors of the denominator of parameter $\alpha$, i.e., $r$. For
example, if we consider the sequence $\forall n\geq1:$
\[
6\times3^{\sum_{j=0}^{\infty}\left\lfloor n/3^{j}\right\rfloor }%
2^{\left\lfloor n/2\right\rfloor }M_{n}(0,1/3,6-1/3)/4^{n},
\]
then the first few elements of this sequence are the following :
\newline$1,8/7,3,20/3,26/3,832/11,3952/33,1216/3,45600/7.$ Hence we have apart
of divisors of $3$ and $6$ we have also $7,$ $11$ and maybe others.
\end{remark}

\begin{remark}
Notice, that in general, the integer sequence
\[
\left\{  \left(  \frac{r}{4}\right)  ^{n}\prod_{j=1}^{k}d_{j}^{\sum
_{m=1}^{\infty}\left\lfloor n/d_{j}^{m}\right\rfloor }M_{n}%
(0,p/r,1-r/r)\right\}  _{n\geq0}%
\]
is not a moment sequence, even though $\left\{  r^{n}\right\}  $ and $\left\{
M_{n}(0,p/r,1-r/r)/4^{n}\right\}  _{n\geq0}$ are. This is so because $\left\{
\prod_{j=1}^{k}d_{j}^{\sum_{m=1}^{\infty}\left\lfloor n/d_{j}^{m}\right\rfloor
}\right\}  _{n\geq0}$ is not a moment sequence. On the other hand sequences
given by (\ref{beta_int}), allow to increase knowledge on some sequence
presented in OEIS. For example sequence (\ref{beta_int}), with $p/r\allowbreak
=\allowbreak1/3$ is listed as A004117 in OEIS or with $p/r\allowbreak
=\allowbreak1/8$ as A181161 in OEIS.
\end{remark}

Now, we will calculate moments $M_{n}(0,\alpha,\beta)$ for all values
$\alpha\allowbreak=\allowbreak n/2$ and $\beta\allowbreak=\allowbreak m/2, $
where $n$ and $m$ are natural numbers.

\section{First, partial results\label{first}}

\begin{lemma}
\label{pomoc} i) For natural $n$ and non-negative integer $i$we have:%
\[
(i+\frac{1}{2})^{(n)}\allowbreak=\allowbreak\frac{(2i+2n)!i!}{4^{n}%
(i+n)!(2i)!},~\left(  \frac{1}{2}\right)  _{(n)}\allowbreak=\allowbreak
(-1)^{n-1}\frac{2(2n-2)!}{4^{n}(n-1)!},
\]
setting $1$ for $n\allowbreak=\allowbreak0$ and for $n\geq1:$%
\begin{align*}
\left(  \frac{3}{2}\right)  _{(n)}\allowbreak &  =\allowbreak\left\{
\begin{array}
[c]{ccc}%
3/2 & \text{if} & n=1\\
(-1)^{n}\frac{12(2n-4)!}{4^{n}(n-2)!} & \text{if} & n>1
\end{array}
\right.  ,\\
\left(  \frac{5}{2}\right)  _{(n)}\allowbreak &  =\allowbreak\left\{
\begin{array}
[c]{ccc}%
5/2 & \text{if} & n=1\\
3/2 & \text{if} & n=2\\
(-1)^{n-3}\frac{120(2n-6)!}{4^{n}(n-3)!} & \text{if} & n>2
\end{array}
\right.  .
\end{align*}

ii) For $n\geq0:$%
\[
M_{n}\left(  0,i+\frac{1}{2},j+\frac{1}{2}\right)  \allowbreak=\allowbreak
\frac{\binom{2i+2n}{i+n}\binom{i+n}{n}}{\binom{2i}{i}\binom{i+j+n}{n}%
}\allowbreak=\allowbreak\frac{(2n+2i)!i!(i+j)!}{(i+n)!(2i)!(i+j+n)!},
\]
hence in particular:

iiA) $M_{n}(0,i+1/2,1/2)\allowbreak=\allowbreak\binom{2n+2i}{n+i}/\binom
{2i}{i}$,\newline iiB) $M_{n}(0,i+1/2,3/2)\allowbreak=\allowbreak
C_{i+n}/C_{i}$, \newline iiC) $M_{n}(0,i+1/2,5/2)\allowbreak=\allowbreak
((i+2)!i!(2n+2i)!)/((2i)!(n+i)!(n+i+2)!)$ \newline iiD) $M_{n}%
(0,i+1/2,7/2)\allowbreak=\allowbreak((i+3)!i!(2n+2i)!)/((2i)!(i+n)!(n+i+3)!),$
\newline iiE) $M_{n}(0,i+1/2,9/2)\allowbreak=\allowbreak
(i!(i+4)!(2n+2i)!)/((2i)!(i+n)!(n+i+4)!),$ \newline iiF) $M_{n}%
(0,1/2,j+1/2)\allowbreak=\allowbreak(j!(2n)!)/(n!(n+j)!),$ \newline iiG)
$M_{n}(0,3/2,j+1/2)\allowbreak=\allowbreak((j+1)!(2n+1)!)/(n!(n+j+1)!),$%
\newline iiH) $M_{n}(0,5/2,j+1/2)\allowbreak=\allowbreak
((j+2)!(2n+3)!)/((n+1)!(n+2+j)!)$, \newline iiI) $M_{n}%
(0,7/2,j+1/2)\allowbreak=\allowbreak((j+3)!(2n+5)!)/(60(n+2)!(n+j+3)!),$
\newline iiJ) $M_{n}(0,9/2,j+1/2)\allowbreak=\allowbreak
(4!(4+j)!(2n+8)!)/(8!(n+4)!(n+4+j)!)$.

iii) For $n\geq0:$%
\[
M_{n}(0,i+1/2,j)\allowbreak=\allowbreak4^{n}\frac{(2i+2j)!i!}{(2i)!(i+j)!}%
\frac{(2i+2n)!(i+j+n)!}{(i+n)!(2i+2j+2n)!},
\]
hence in particular we get:

iiiA) $M_{n}(0,1/2,1)\allowbreak=\allowbreak4^{n}/(2n+1),$

iiiB)$M_{n}(0,3/2,1)\allowbreak=\allowbreak3\times4^{n}/(2n+3),$

iiiC) $M_{n}(1/2,2)\allowbreak=\allowbreak3\times4^{n}/((2n+1)(2n+3)),$

iiiD) $M_{n}(0,3/2,2)\allowbreak=\allowbreak4^{n}15/((2n+3)(2n+5)).$

iv) For $n\geq0:$%
\[
M_{n}(0,i,j+1/2)\allowbreak=\allowbreak4^{2n}\frac{(i+n-1)!(i+j+n)!(2i+2j)!}%
{(2i+2j+2n)!(i-1)!(i+j)!},
\]
hence in particular we have:\newline ivA) $M_{n}(0,1,1/2)\allowbreak
=\allowbreak4^{2n}(n!)^{2}/(2n+1)!,$ \newline ivB) $M_{n}(0,1,3/2)\allowbreak
=\allowbreak4^{2n}12n!(n+2)!/(2n+4)!,$ \newline ivC) $M_{n}%
(0,2,1/2)\allowbreak=\allowbreak4^{2n}12(n+1)!(n+2)!/(2n+4)!,$\newline ivD)
$M_{n}(0,2,3/2)\allowbreak=\allowbreak$ $4^{2n}120(n+1)!(n+3)!/(2n+6)!,$
\newline ivE) $M_{n}(0,3,1/2)\allowbreak=\allowbreak60\times4^{n}%
(n+2)!(n+3)!/(2n+6)!$, \newline ivF) $M_{n}(0,4,1/2)\allowbreak=\allowbreak
4^{2n}280(n+4)!(n+3)!/(2n+8)!.$
\end{lemma}

\begin{proof}
i) Following definitions given by (\ref{down}) and (\ref{up}) we get
$(i+1/2)^{(n)}\allowbreak=\allowbreak(i+1/2)(i+3/2)\ldots
(i+1/2+n-1)\allowbreak=\allowbreak(2n+2i-1)!!/(2^{n}(2i-1)!!)$. Now, it is
elementary to check that $(i+1/2)^{(n)}\allowbreak=\allowbreak\frac
{(2i+2n)!i!}{4^{n}(i+n)!(2i)!}$. Generally, we have $(i+1/2)_{n}%
\allowbreak=\allowbreak(-1)^{n}(-i-1/2)^{n},$ hence we could use this formula
to get $(1/2)_{n},$ $(3/2)_{n}$ and $(5/2)_{n}$ but it seems that it might be
easier to check these formulae directly. ii) Applying assertion i) twice we
get ii). iii) and iv) we apply assertion i) once but in the case of iii) in
the numerator and in the case of iv) in denominator. v) Is direct application
of (\ref{Mm}) and (\ref{iMm}).
\end{proof}

Some sequences of integers, that will appear in the sequel, are identifiable
by their generating functions in the basis OEIS, hence we need to calculate
also the generating functions of the moments of the Jacobi distributions that
will appear in the sequel. In fact, there are two ways of calculating the
generating functions of moment sequences. The first one, so to say, direct,
uses a formula (\ref{mJ}) and the other uses the integral representation of
the moment sequence. The Lemma \ref{gen}, above lists some of the generating
functions and presents some of the ways to transform them.

Thus, as a corollary we have:

\begin{corollary}
\label{gf}Let us denote by $G(x;c,\alpha,\beta)\allowbreak=\allowbreak
\sum_{k\geq0}x^{k}M_{k}(c,\alpha,\beta)$. Then we have for $c\in\mathbb{R}$,
$\alpha,\beta>0$ $G(0,c,\alpha,\beta)\allowbreak=\allowbreak1$ while for
$x\neq0$ and $\left\vert x\right\vert <1/4$ we get:

a) $G(x;0,1/2,1/2)\allowbreak=\allowbreak1/\sqrt{1-4x}$,

b) $G(x;0,3/2,1/2)\allowbreak=\allowbreak(1-\sqrt{1-4x}/(2x\sqrt{1-4x})$,

c) $G(x;0,1/2,3/2)\allowbreak=\allowbreak\left(  1-\sqrt{1-4x}\right)  /(2x)$,

d) $G(x;0,1/2,5/2)\allowbreak=\allowbreak\left(  (1-4x)^{3/2}+6x-1\right)
/\left(  6x^{2}\right)  $,

e) $G(x;0,3/2,3/2)\allowbreak=\allowbreak\left(  1-2x-\sqrt{1-4x}\right)
/\left(  2x^{2}\right)  $,

f) $G(x;0,5/2,1/2)\allowbreak=\allowbreak\left(  1-\sqrt{1-4x}-2x\sqrt
{1-4x}\right)  /\left(  6x^{2}\sqrt{1-4x}\right)  $,

g) $G(x;0,1/2,7/2)\allowbreak=\allowbreak\left(  1-(1-4x)^{5/2}-10x+30x^{2}%
-20x^{3}\right)  /\left(  10x^{4}\right)  $,

h) $G(x;0,3/2,5/2)\allowbreak=\allowbreak\left(  (1-4x)\sqrt{1-4x}%
-1+6x-6x^{2}\right)  /\left(  4x^{3}\right)  $,

i) $G(x;0,5/2,3/2)\allowbreak=\allowbreak\left(  -1+(1-4x)^{1/2}%
+2x+2x^{2}\right)  /\left(  4x^{3}\right)  $,

j) $G(x;0,7/2,1/2)\allowbreak=\allowbreak\left(  1-(1-4x)^{1/2}(1+2x+6x^{2}%
)\right)  /\left(  20x^{3}\sqrt{1-4x}\right)  $,
\end{corollary}

In the presented below, identification of particular moments sequences in the
basis OEIS, we will use the following terminology: the sequence A is equal to
the $l-s(p_{1},p_{2}\ldots)$ sequence B meaning that the sequence A is
obtained from the sequence B by omitting the first elements that are equal to
$p_{1},p_{2},\ldots$. Similarly, the sequence A is $r-s(p_{1},p_{2},\ldots)$
sequence B means sequence the B is obtained from sequence A by adding at its
beginning numbers $p_{1},p_{2},\ldots$. Finally, if sequence the B is obtained
from sequence the A by sign-change briefly $sc$ if even elements of both
sequences are the same while odd elements have different signs by the same
absolute values.

\begin{remark}
In particular
\begin{equation}
\sum_{k\geq0}C_{n}/4^{n}=2. \label{2}%
\end{equation}
This is so since $C_{n}\allowbreak\cong\allowbreak4^{n}/n^{3/2}$ ,thus the
series above is convergent. Its sum can be found easily by passing to the
limit $x\rightarrow1/4$ in assertion $c$ of the Corollary \ref{gf}.
\end{remark}

\begin{remark}
\label{Pocz}a) $M_{n}(0,1/2,1/2)\allowbreak=\allowbreak\binom{2n}{n}$,
sequence A000984 in OEIS. It is called the sequence of "Central binomial",

b) $M_{n}(0,1/2,3/2)\allowbreak=\allowbreak\binom{2n}{n}/(n+1)$, sequence
A000108 in OEIS. It is called the sequence of "Catalan numbers",

c) $M_{n}(0,3/2,1/2)\allowbreak=\allowbreak\binom{2n+1}{n+1}\allowbreak
=\allowbreak\binom{2n+2}{n+1}/2,$ sequence A001700 in OEIS,

d) $M_{n}(0,1/2,5/2)\allowbreak=\allowbreak\left(  2(2n)!\right)  /\left(
n!(n+2)!\right)  $, $\frac{1}{3}\times$ "super ballot numbers" i.e., $\frac
{1}{3}\times$ of the sequence A007054 in OEIS,

e) $M_{n}(0,3/2,3/2)\allowbreak=\allowbreak\binom{2n+2}{n+1}/(n+2),$ $l-s(1)$
Catalan numbers', i.e., $l-s(1)$ i.e., sequence A000108 in OEIS,

f) $M_{n}(0,5/2,1/2)\allowbreak=\allowbreak\binom{2n+4}{n+2}/6,$ $\frac{1}{3}$
$l-s(1)$ sequence A001700 in OEIS,

g) $M_{n}(0,1/2,7/2)\allowbreak=\allowbreak\left(  6(2n+2)!\right)  /\left(
(n+1)!(n+4)!\right)  $,$\frac{1}{10}\times r-s(10)$ super ballot numbers i.e.,
$\frac{1}{10}\times r-s(10)$ of the sequence A007272 in OEIS,

h) $M_{n}(0,3/2,5/2)\allowbreak=\allowbreak\left(  3(2n+2)!\right)  /\left(
(n+1)!(n+3)!\right)  ,\frac{1}{2}\times r-s(3)$ super ballot numbers more
precisely $\frac{1}{2}\times r-s(3$ ) sequence A007054 in OEIS,

i) $M_{n}(0,5/2,3/2)\allowbreak=\allowbreak\binom{2n+4}{n+2}/\left(  2\left(
n+3\right)  \right)  ,\frac{1}{2}\times r-s(1,1)$ Catalan numbers,

j) $M_{n}(0,7/2,1/2)\allowbreak=\allowbreak\left(  (2n+6)!\right)  /\left(
20(n+3)!(n+3)!\right)  $, $\frac{1}{10}\times r-s(1,3)$ sequence A001700 in OEIS.
\end{remark}

\begin{corollary}
\label{bc}1) $4^{n}M_{n}(-3/4,1/2,1/2)$ is the sequence A322248 in OEIS, since
the g.f. of this sequence is $1/\sqrt{\left(  1-13x\right)  \left(
1+3x\right)  }$ which can be obtained by the formula given in assertion iii)
of Lemma \ref{gen}.

2) $4^{n}M_{n}(-7/4,1/2,1/2)$ is the sequence A098441 in OEIS, since the g.f.
of this sequence is $1/\sqrt{1-2x-63x^{2}}$ which can be obtained by the
formula given in assertion iii) of Lemma \ref{gen}.

3) $2^{n}M_{n}(-3/2,1/2,1/2)$ is the sequence A084605 in OEIS, since the g.f.
of this sequence is $1/\sqrt{1-2x-15x^{2}}$ which can be obtained by the
formula given in assertion iii) of Lemma \ref{gen}.

4) $2^{n}M_{n}(-3/2,1/2,1/2)$ is the sequence A084605 in OEIS, since the g.f.
of this sequence is $1/\sqrt{1-2x-15x^{2}}$ which can be obtained by the
formula given in assertion iii) of Lemma \ref{gen}.

5) $M_{n}(-1,1/2,1/2)$ is the sequence A002426 'central trinomial coefficient'
in OEIS, by assertion i) of Lemma \ref{binomial}.

6) $2^{n}M_{n}(-1/2,1/2,1/2)$ is the sequence A322242 in OEIS, since the g.f.
of this sequence is $1/\sqrt{1-6x-7x^{2}}$ which can be obtained by the
formula given in assertion iii) of Lemma \ref{gen}.

7) $4^{n}M_{n}(-1/4,1/2,1/2)$ is not listed in OEIS. The g.f. of this sequence
is \newline$1/\sqrt{1-14x-15x^{2}}$ which can be obtained by the formula given
in assertion iii) of Lemma \ref{gen}. Visibly this sequence is closely related
to sequence A098441 in OEIS by appropriate binomial transform (Lemma
\ref{binomial}).

8) $4^{n}M_{n}(1/4,1/2,1/2)$ is not listed in OEIS. The g.f. of this sequence
is $1/\sqrt{1-18x+17x^{2}}$ which can be obtained by the formula given in
assertion iii) of Lemma \ref{gen}. Visibly this sequence is closely related to
sequence A098441 in OEIS by appropriate binomial transform (Lemma
\ref{binomial}). .

9) $2^{n}M_{n}(1/2,1/2,1/2)$ is the sequence A084771 in OEIS, since the g.f.
of this sequence is $1/\sqrt{1-10x+9x^{2}}$ which can be obtained by the
formula given in assertion iii) of Lemma \ref{gen}.

10) $M_{n}(1,1/2,1/2)$ is the sequence A026375 in OEIS, since the g.f. of this
sequence is $1/\sqrt{1-6x+5x^{2}}$ which can be obtained by the formula given
in assertion iii) of Lemma \ref{gen}.

11) $4^{n}M_{n}(5/4,1/2,1/2)$ is not listed in OEIS. The g.f. of this sequence
is $1/\sqrt{1-26x+105x^{2}}$ which can be obtained by the formula given in
assertion iii) of Lemma \ref{gen}. Visibly this sequence is closely related to
sequence A098441 in OEIS by appropriate binomial transform (Lemma
\ref{binomial}).

12) $2^{n}M_{n}(3/2,1/2,1/2)$ is the sequence A248168 in OEIS, since the g.f.
of this sequence is $1/\sqrt{1-14x+33x^{2}}$ which can be obtained by the
formula in assertion iii) of Lemma \ref{gen}. given in assertion iii) of Lemma
\ref{gen}.

13) $M_{n}(2,1/2,1/2)$ is the sequence A081671 in OEIS, since the g.f. of this
sequence is $1/\sqrt{1-8x+12x^{2}}$ which can be obtained by the formula given
in assertion iii) of Lemma \ref{gen}.

14) $M_{n}(2,3/2,1/2)$ is the $l-s(1)$ sequence A005573 in OEIS since the g.f.
of this sequence is $\left(  \sqrt{1-2x}-\sqrt{1-6x}\right)  /(2x\sqrt{1-6x})$
which can be obtained by the formula given in assertion iii) of Lemma
\ref{gen}.

15) $2^{n}M_{n}(5/2,1/2,1/2)$ is not listed in OEIS. The g.f. of this sequence
is $1/\sqrt{1-18x+65x^{2}}$ that can be obtained by the formula given in
assertion iii) of Lemma \ref{gen}. Visibly this sequence is closely related to
sequence A084771 in OEIS by appropriate binomial transform (Lemma
\ref{binomial}).
\end{corollary}

Now, we examine some examples connected with Catalan numbers.

\begin{corollary}
\label{catalan} 1) $M_{n}(-1,1/2,3/2)$ 'Riordan numbers' i.e., sequence
A005043 in OEIS. This is so since by Lemma \ref{gen} its generating function
is equal to \newline$(1\allowbreak-\allowbreak\sqrt{(1-3x)/(1+x)})/(2x)$ as
given in unnumbered formula on page 87 of \cite{Bern99}.

2) $M_{n}(-1,3/2,3/2)$ sequence A001006 in OEIS the so-called Motzkin numbers.
This is so since its g.f. is $f(x)\allowbreak=\allowbreak(1-x-\sqrt
{1-2x-3x^{2}})/(2x^{2}),$ that satisfies the following equation $x^{2}%
f^{2}\allowbreak+\allowbreak(x-1)f\allowbreak+\allowbreak1\allowbreak
=\allowbreak0$ as given in Wikipedia.

3) $2^{n}M_{n}(-1/2,1/2,3/2)$ is the sequence A337168 in OEIS, since g.f. of
this sequence is $A(x)\allowbreak=\allowbreak\left(  -1+\sqrt{(1-7x)/(1+x)}%
\right)  /\left(  4x\right)  $ that can be obtained by the formula given in
assertion iii) of Lemma \ref{gen}. (There is small misprint in the formula for
g.f. in OEIS (there is $8$ instead of $7).$ However by simple check one can
see that the $A(x)$ satisfies the following identity $A(x)\allowbreak
=\allowbreak1/(x+1)+2xA(x)^{2}$).

4) $2^{n}M_{n}(-3/2,1/2,3/2)$ is not listed in OEIS, it's g.f. of this
sequence is $\left(  -1+\sqrt{(1-5x)/(1+3x)}\right)  /\left(  4x\right)  $
that can be obtained by the formula given in assertion iii) of Lemma
\ref{gen}. Visibly this sequence is closely related to sequence A337168 in
OEIS by appropriate binomial transform. Compare Lemma \ref{binomial}.

5) $2^{n}M_{n}(1/2,1/2,3/2)$ is the $r-s(1)$ sequence A162326 in OEIS. This is
so since g.f. of $2^{n}M_{n}(1/2,1/2,3/2)$ is $(1\allowbreak-\allowbreak
\sqrt{(1-9x)/(1-x)})/(4y)$ hence the g.f. of the sequence A162326 is
$1\allowbreak+\allowbreak x(1\allowbreak-\allowbreak\sqrt{(1-9x)/(1-x)}%
)/(4y)\allowbreak=\allowbreak(5\allowbreak-\allowbreak\sqrt{(1-9x)/(1-x)})/4$.

6) $M_{n}(1,1/2,3/2)$ sequence A007317 in OEIS defined as 'binomial transform
of Catalan numbers' and its indeed since we have Lemma \ref{binomial} and
following it remark.

7) $M_{n}(1,3/2,1/2)$ is not listed in OEIS, However by (\ref{Mm}) we see that
it is a binomial transform of the sequence A001700 of OEIS.

8) $2^{n}M_{n}(3/2,1/2,3/2)$ sequence not listed in OEIS. Its g.f. is equal to
$(1-\sqrt{\left(  1-11x\right)  /(1-3x)}/(4x)$. Visibly this sequence is
closely related to sequence A162326 in OEIS by appropriate binomial transform.
Compare Lemma \ref{binomial}.

9) $M_{n}(2,1/2,3/2)$ sequence A064613 in OEIS defined as ' second binomial
transform of Catalan numbers' its g.f. is $(1-\sqrt{(1-6x)/(1-2x)})/(2x).$
\end{corollary}

Now, we will present some assorted examples that seem to be important from the
point of view of combinatorics.

\begin{corollary}
1) $M_{n}(1,3/2,3/2)$ is the $r-s(1)$ sequence A002212 in OEIS .

2) $M_{n}(2,3/2,3/2)$ sequence A005572 in OEIS, \newline since its g.f. is
$(1\allowbreak-\allowbreak4x\allowbreak-\allowbreak\sqrt{1-8x+x^{2}}%
)/(2x^{2})$

3) $M_{n}(3,3/2,3/2)$ sequence A182401 in OEIS, \newline since its g.f. is
$(1\allowbreak-\allowbreak5x\allowbreak-\allowbreak\sqrt{1-10x+21x^{2}%
})/(2x^{2})$

4) $M_{n}(1/2,3/2,3/2)$ is the $r-s(1)$ sequence A059231 in OEIS, \newline
since its g.f. is $f(x)\allowbreak=\allowbreak(1\allowbreak-\allowbreak
5x\allowbreak-\allowbreak\sqrt{1-10x+9x^{2}})/(8x^{2}),$ \newline hence
$1\allowbreak+\allowbreak xf(x)\allowbreak=\allowbreak\allowbreak
(1\allowbreak+\allowbreak3x\allowbreak-\allowbreak\sqrt{1-10x+9x^{2}})/(8x).$
\end{corollary}

Sometimes the description of the sequence in the OEIS is insufficient. Either
the sequence generator function is not given, or the formula for the $n$-the
item in the sequence is missing. However, in some cases, one observes that the
first several elements of the sequence $\{M_{n}(c,\alpha,\beta)\}$ for some
values of parameters $c,\alpha,\beta$ agree with the elements of the sequence
from the base OEIS. Then there exists a strong supposition that these two
sequences are identical. We will present a few of these suppositions in the
form of conjectures, presented below.

\begin{conjecture}
1. $M_{n}(-1,3/2,1/2)$ is the sequence A005773 of OEIS. More precisely it is
an inverse Binomial transform of the sequence A001700.

2. $2^{n}M_{n}(-3/2,3/2,1/2)$ is the sequence A151318 of OEIS.

3. $2M_{n}(-1,5/2,3/2)$ is the sequence A005554 in OEIS.

4. $2M_{n}(1,5/2,3/2)$ is the sequenceA045868 in OEIS.
\end{conjecture}

Notice, that the distribution with the density $g(x;-2,\alpha,\alpha)$ is
symmetric hence all odd moments of the form $M_{n}(-2,\alpha,\alpha)$ are
equal to $0$. Besides the support of this distribution is symmetric $[-2,2].$
All these reasons suggest that it should be in a special way. Namely, we will
denote by%
\begin{align*}
g(y;-2,\gamma,\delta)\overset{def}{=}b(y;\gamma,\delta) &  =\frac
{(y+2)^{\gamma-1}(2-y)^{\delta-1}}{4^{\gamma+\delta-1}B(\gamma,\delta)},\\
S_{n}(\gamma,\delta)\allowbreak &  =\allowbreak\int_{-2}^{2}x^{n}%
b(x;\gamma,\delta)dx.
\end{align*}

\begin{remark}
\label{symm}i) Notice, also, that we have for all $n\geq0,$ $\gamma,\delta>0:$%
\[
S_{n}(\delta,\gamma)\allowbreak=\allowbreak(-1)^{n}S_{n}(\gamma,\delta).
\]
This is so since we have for $\gamma,\delta>0$ and $x\in\lbrack-2,2]:$
\[
b(x;\gamma,\delta)\allowbreak=\allowbreak b(-x;\delta,\gamma).
\]
Using Lemma \ref{binomial}, we have:%
\begin{align*}
S_{n}(\gamma,\delta)  &  =\sum_{j=0}^{n}\binom{n}{j}(-2)^{n-j}M_{j}%
(0,\gamma,\delta),\\
M_{n}(0,\gamma,\delta)  &  =\sum_{j=0}^{n}\binom{n}{j}2^{n-j}S_{j}%
(\gamma,\delta).
\end{align*}

\end{remark}

As far as the values of the moments $S_{n}$ are concerned, we have:

\begin{proposition}
\label{Cent}a) $S_{n}(1/2,1/2)=\left\{
\begin{array}
[c]{ccc}%
0 & \text{if} & n\text{ is odd}\\
\binom{n}{n/2} & \text{if} & n\text{ is even}%
\end{array}
\right.  $, sequence A126869 in OEIS,

b) $S_{n}(1,1)=\left\{
\begin{array}
[c]{ccc}%
0 & \text{if} & n\text{ is odd}\\
\frac{2^{n}}{(n+1)} & \text{if} & n\text{ is even}%
\end{array}
\right.  $,

c) $S_{n}(3/2,3/2)\mathbb{\allowbreak}\mathbb{=\allowbreak}\left\{
\begin{array}
[c]{ccc}%
0 & \text{if} & n\text{ is odd}\\
C_{n/2} & \text{if} & n\text{ is even}%
\end{array}
\right.  $, sequence A126120 in OEIS,

d) $S_{n}(2,2)=\left\{
\begin{array}
[c]{ccc}%
0 & \text{if} & n\text{ is odd}\\
\frac{3\times2^{n}}{(n+1)(n+3)} & \text{if} & n\text{ is even}%
\end{array}
\right.  $,

e) $S_{n}(1/2,3/2)\allowbreak=\allowbreak(-1)^{n}\binom{n}{\left\lfloor
n/2\right\rfloor }$, sequence A126930 in OEIS, since its g.f. is
$(1\allowbreak-\allowbreak\sqrt{(1-2x)/(1+2x)})/(2x)$, while $S_{n}%
(3/2,1/2)\allowbreak=\allowbreak\binom{n}{\left\lfloor n/2\right\rfloor }.$
Strangely it has different number in OEIS (A001405),

f) $S_{n}(1,2)\allowbreak=\allowbreak\left\{
\begin{array}
[c]{ccc}%
\frac{-2^{n}}{n+2} & \text{if} & n\text{ is odd}\\
\frac{2^{n}}{n+1} & \text{if} & n\text{ is even}%
\end{array}
\right.  $, while $S_{n}(2,1)\allowbreak=\allowbreak=\left\{
\begin{array}
[c]{ccc}%
\frac{2^{n}}{n+2} & \text{if} & n\text{ is odd}\\
\frac{2^{n}}{n+1} & \text{if} & n\text{ is even}%
\end{array}
\right.  $, by Remark \ref{symm}.

g) $2S_{n}(3/2,5/2)$ is the $r-s(1,1)$ and $sc$ sequence A089408 in OEIS since
the g.f. of $2S_{n}(3/2,5/2)$ is $g(x)\allowbreak=\allowbreak(-6x^{2}%
\allowbreak+\allowbreak6x\allowbreak-\allowbreak1\allowbreak+\allowbreak
(1-4x)^{3/2})/2x^{3}$ by Lemma \ref{gen}. Operation $r-s(1,1)$ changes this
function to $1-x\allowbreak+\allowbreak x^{2}g(x)\allowbreak=\allowbreak
(4x\allowbreak-\allowbreak1\allowbreak+\allowbreak(1-2x)\sqrt{1-4x^{2}}%
)/(2x)$. Now it remains to change $x$ to $-x.$ By Remark \ref{symm}
$2S_{n}(5/2,3/2)$ is $r-s(1,1)$ sequence A089408 in OEIS, since following its
description remains to change $x$ to $-x$ in $(4x\allowbreak-\allowbreak
1\allowbreak+\allowbreak(1-2x)\sqrt{1-4x^{2}})/(2x)$.
\end{proposition}

\begin{remark}
Taking into account Corollary \ref{catalan}2) identifying Motzkin numbers by
their generating function, identity \ref{MbMc} with $b\allowbreak
=\allowbreak-1$ and $c\allowbreak=\allowbreak-2$ and assertion c) of
Proposition \ref{Cent}, we immediately arrive to the formula relating Motzkin
numbers and Catalan numbers, namely:%
\begin{equation}
Mo_{n}\allowbreak=\allowbreak\sum_{j=0}^{\left\lfloor n/2\right\rfloor }%
\binom{n}{2j}C_{j}, \label{Motz}%
\end{equation}
where $Mo_{n}$ denotes $n-th$ Motzkin number. The formula above relating
Motzkin and Catalan numbers appears in the solution of problem 4 in the
Stanley's book \cite{Stan15}.
\end{remark}

\section{Expansions\label{expan}}

In this section, we are going to get some interesting identities involving the
mentioned-above moments similar to the ones presented in the introduction.
These identities will be obtained by simple expansions of the ratio of the
densities of the involved beta distributions. To avoid unnecessary
complications we will apply these expansions only in two cases. Namely, let us
take real $\gamma$ such that $\alpha-\gamma$ is an integer then we can expand
the ratio of $b(x;\alpha,\beta)/b(x;\gamma,\delta)$ in the following series:%
\[
\frac{b(x;\alpha,\beta)}{b(x;\gamma,\delta)}=\frac{B(\gamma,\delta
)4^{2\delta-\beta-1}}{B(\alpha,\beta)}\sum_{k\geq0}\frac{x^{\alpha-\gamma+k}%
}{4^{k+\alpha-\gamma}}\frac{(\beta-\delta)_{k}}{k!}.
\]
and consequently, we can relate moment sequences of the two distributions one
with the density $b(x;\alpha,\beta)$ and the other with the density
$b(x;\alpha,\beta)$.%

\begin{equation}
M_{n}(0,\alpha,\beta)=\frac{B(\gamma,\delta)4^{\gamma-\alpha+\delta-\beta}%
}{B(\alpha,\beta)}\sum_{k\geq0}\frac{(\beta-\delta)_{k}}{k!4^{k+\alpha-\gamma
}}M_{k+n+\alpha-\gamma}(0,\gamma,\delta). \label{exp}%
\end{equation}
We will also exploit the similar following trick, related to the expansion of
the ratio $g(x;\alpha,\beta)/g\left(  x;\gamma,\delta\right)  .$ Namely, we
consider the following expansion:%
\begin{equation}
b\left(  x;\alpha,\beta\right)  \allowbreak=\allowbreak\frac{B(\gamma,\delta
)}{4^{\alpha-\gamma+\beta-\delta}B(\alpha,\beta)}b(x;\gamma,\delta)\sum
_{k\geq0}\frac{x^{k}}{4^{k}k!}c_{k}(\alpha-\gamma,\beta-\delta), \label{exp2}%
\end{equation}
based on the fact that
\[
g(x;\alpha,\beta)/g\left(  x;\gamma,\delta\right)  \allowbreak=\allowbreak
\frac{B(\gamma,\delta)}{4^{\alpha-\gamma+\beta-\delta}B(\alpha,\beta
)}(y+2)^{\alpha-\gamma}(2-y)^{\beta-\delta},
\]
expansion (\ref{bin}) and the standard multiplication of power series. Thus we
have:
\begin{equation}
c_{k}(\alpha-\gamma,\beta-\delta)=\sum_{j=0}^{k}\binom{k}{j}(-1)^{k-j}%
(\alpha-\gamma)_{j}(\beta-\delta)_{k-j}. \label{ck}%
\end{equation}
(\ref{exp2}) leads to the following expansion involving moments $S_{n}.$%
\begin{equation}
S_{n}(\alpha,\beta)=\frac{B(\gamma,\delta)}{4^{\alpha-\gamma+\beta-\delta
}B(\alpha,\beta)}\sum_{k\geq0}\frac{c_{k}(\alpha-\gamma,\beta-\delta)}%
{4^{k}k!}S_{n+k}(\gamma,\delta). \label{exp3}%
\end{equation}

As far as the question of convergence of the series (\ref{exp}) and
(\ref{exp3}) is concerned, we have the following remarks.

\begin{proposition}
\label{zbi}$\forall n\geq0$:

i) $\left\vert M_{n}(0,\alpha,\beta)\right\vert <4^{n},$ $\left\vert
S_{n}(\alpha,\beta)\right\vert <2^{n}$,

ii) $\left\vert M_{n}(0,i+1/2,j+1/2)\right\vert \allowbreak=\allowbreak
4^{n}O(1/n^{j+1/2}),$

iii) $\left\vert M_{n}(0,i+1/2,j)|\allowbreak=\allowbreak|4^{n}\left(
(2n+2i)!i!(i+j)!\right)  /\left(  (i+n)!(2i)!(i+j+n)!\right)  \right\vert
\allowbreak$\newline$=\allowbreak4^{2n}O(1/n^{j+1/2})$

iv) $\left\vert M_{n}(0,i,j+1/2)\allowbreak\right\vert =\allowbreak$%
\newline$\left\vert 4^{2n}\left(  (i+n-1)!(i+j+n)!(2i+2j)!\right)  /\left(
(2i+2j+2n)!(i-1)!(i+j)!\right)  \right\vert $\newline$\allowbreak
=\allowbreak4^{n}O(1/n^{j+1/2}).$
\end{proposition}

\begin{proof}
i) This fact follows directly from the supports of the measures whose moments
are considered and the fact that the measures concerned are absolutely
continuous with respect to Lebesgue measure, i.e., have densities. Namely, in
the first case it is the segment $[0,4]$ while in the second case the segment
$[-2,2].$ For ii), iii) and iv) we start with the fact that $\binom{2n}%
{n}\frac{1}{4^{n}}\cong O(1/n^{1/2})$ as $n\rightarrow\infty$. Secondly, let
us notice that $(n+k)!/n!\allowbreak=(n+1)\allowbreak\ldots(n+k)\allowbreak
=\allowbreak O(n^{k}).$ Hence, for the case ii) we have
\[
\left\vert \frac{(2n+2i)!i!(i+j)!}{(i+n)!(2i)!(i+j+n)!}\right\vert
\allowbreak\cong\allowbreak\frac{i!(i+j)!}{(2i)!}\frac{4^{n+i}}{\sqrt{n+i}%
}\frac{(n+i)!}{(n+i+j+n)!}\allowbreak\cong\allowbreak O(4^{n}/n^{j+1/2}),
\]
iii)%
\[
\left\vert 4^{n}\frac{(2n+2i)!i!(i+j)!}{(i+n)!(2i)!(i+j+n)!}\right\vert
\allowbreak\cong\allowbreak4^{n}\frac{i!(i+j)!}{(2i)!}\binom{2n+2i}{n+i}%
\frac{(n+i)!}{(i+j+n)!}\allowbreak\cong\allowbreak4^{2n}O(1/n^{j+1/2}),
\]
iv) $\allowbreak$%
\begin{align*}
\left\vert 4^{2n}\frac{(i+n-1)!(i+j+n)!(2i+2j)!}{(2i+2j+2n)!(i-1)!(i+j)!}%
\right\vert \allowbreak &  \cong\\
\allowbreak4^{n}\frac{(2i+2j)!}{(i-1)!(i+j)!}\left(  4^{n}/\binom
{2i+2j+2n}{i+j+n}\right)  \frac{(i+n-1)!}{(i+j+n)!}\allowbreak &
\cong\allowbreak O(1/n^{j+1-1/2}).
\end{align*}

\end{proof}

\begin{theorem}
\label{c=0}For $n\geq0$ we have:

i)
\begin{equation}
C_{n}=2\binom{2n}{n}-\frac{1}{2}\binom{2(n+1)}{n+1}. \label{Cnabin}%
\end{equation}

ii)
\begin{align}
C_{2n+1}  &  =\sum_{i=0}^{n}\binom{2n}{2i}4^{n-i}C_{i},\label{niep}\\
C_{2n+2}  &  =2\sum_{i=0}^{n}\binom{2n+1}{2i}4^{n-i}C_{i} \label{parz}%
\end{align}

iii)%
\[
C_{n}=\frac{3}{2}\sum_{i\geq0}\frac{1}{4^{i}}\frac{(2n+2i)!}{(i+n)!(n+i+2)!},
\]

iv)%
\[
\frac{(n+1)!(n+2)!}{(2n+4)!}=\frac{1}{4^{n+2}}\sum_{j\geq0}\binom{2j}{j}%
\frac{1}{4^{j}(n+j+2)}.
\]
Or equivalently%
\[
\frac{4^{n+2}}{C_{n+2}}=\sum_{j\geq0}\binom{2j}{j}\frac{1}{4^{j}}%
\frac{(n+2)(n+3)}{(2n+2j+2)}.
\]

v)
\[
\frac{1}{4^{n}}\binom{2n}{n}=1-\frac{1}{2}\sum_{j=0}^{n-1}\frac{C_{j}}{4^{j}%
},
\]

vi)%
\[
\frac{n!n!}{(2n+1)!}=4\sum_{j\geq0}4^{j}\frac{(n+j)!(n+j+2)!}{(2n+2j+4)!}.
\]
Or equivalently%
\[
\frac{4^{n}}{C_{n}}=\sum_{j\geq0}\frac{4^{n+j+1}}{C_{n+j+1}}\frac
{(2n+2)(n+1)}{(n+j+1)(2n+2j+3)(2n+2j+4)}.
\]

\end{theorem}

\begin{proof}
i)We take $\alpha\allowbreak=\allowbreak1/2,$ $\beta\allowbreak
=3/2,\allowbreak\gamma\allowbreak=\allowbreak\delta\allowbreak=\allowbreak
1/2.$ Then we apply (\ref{exp}). ii) First we note that $M_{n}%
(0,3/2,3/2)\allowbreak=\allowbreak C_{n+1}$, by Remark \ref{Pocz}e), while
$S_{n}(3/2,3/2)\allowbreak=\allowbreak0$ for $n$ odd and $C_{n/2}$ when $n$ is
even. Then we apply (\ref{iMm} and take $k\allowbreak=\allowbreak2n$ getting
directly (\ref{niep}). We get (\ref{parz}) likewise. iv) We take
$\alpha\allowbreak=\allowbreak1/2$, $\beta\allowbreak=\allowbreak3/2,$
$\gamma\allowbreak=\allowbreak1/2,$ $\delta\allowbreak=\allowbreak5/2$. Then
$\frac{a(x;\alpha,\beta)}{a(x;\gamma,\delta)}\allowbreak=\allowbreak\frac
{3}{4-x}\allowbreak=\allowbreak\frac{3}{4}\sum_{k\geq0}\frac{x^{k}}{4^{k}}$.
Now recall that $M_{n}(0,1/2,3/2)\allowbreak=\allowbreak C_{n}$ and
$M_{n}(0,1/2,5/2)\allowbreak=\allowbreak\frac{2(2n)!}{n!(n+2)!}$ by Remark
\ref{Pocz} b) and d). Now it remains to apply (\ref{exp}). iv) We consider
$\alpha\allowbreak=\allowbreak2$, $\beta\allowbreak=\allowbreak1/2$ and
$\gamma\allowbreak=\allowbreak3/2$ and $\delta\allowbreak=\allowbreak1.$ Then
we apply (\ref{exp}) using expansion
\begin{align*}
\frac{g(x;0,2,1/2)}{g(x;0,3/2,1)}  &  =\frac{1}{2}\sqrt{\frac{x}{4-x}}%
=\frac{1}{4}\sum_{j\geq0}(-1)^{j}(-1/2)_{(j)}\frac{x^{j+1/2}}{4^{j}j!}\\
&  =\frac{1}{4}\sum_{j\geq0}\frac{(2j)!}{4^{j}j!}\frac{x^{j+1/2}}{4^{j}%
j!}=\frac{1}{4}\sum_{j\geq0}\binom{2j}{j}\frac{1}{4^{2j}}x^{j+1/2}.
\end{align*}
Now we notice that $\int_{0}^{4}x^{n+1/2}g(x;0,3/2,1)dx\allowbreak
=\allowbreak\frac{3}{2}\int_{0}^{4}x^{n}g(x;0,2,1)dx\allowbreak=\allowbreak
\frac{3}{2}4^{n}\frac{(2)^{(n)}}{(3)^{(n)}}\allowbreak=\allowbreak\frac{3}%
{2}4^{n}\frac{2(n+1)!}{(n+2)!}\allowbreak=\allowbreak\frac{3\times4^{n}}{n+2}%
$. Hence%
\begin{align*}
4^{2n}\frac{12(n+1)!(n+2)!}{(2n+4)!}  &  =\frac{1}{4}\sum_{j\geq0}\binom
{2j}{j}\frac{1}{4^{2j}}\int_{0}^{4}x^{j+n+1/2}g(x;0,3/2,1)dx\\
&  \frac{1}{4}\sum_{j\geq0}\binom{2j}{j}\frac{1}{4^{2j}}\frac{3\times4^{n+j}%
}{(n+j+2)}.
\end{align*}

v) and vi) we apply assertion iii)(iiib) of Lemma \ref{binomial} first for
$\alpha\allowbreak=\allowbreak3/2,$ $\beta\allowbreak=\allowbreak1/2$ and
$\gamma\allowbreak=\allowbreak3/2$ and $\delta\allowbreak=\allowbreak3/2$ and
secondly for $\alpha\allowbreak=\allowbreak1,$ $\beta\allowbreak
=\allowbreak1/2$ and $\gamma\allowbreak=\allowbreak1$ and $\delta
\allowbreak=\allowbreak3/2$ then we use identity (\ref{2}) in the case vi) and
divide both sides by $4^{2n}$ in the case of vii).
\end{proof}

\begin{remark}
Assertion i) is well-known. It appears for example as exercise 4 on page 155
of \cite{Grimaldi12}. It is also easy to get from the assertion i) the formula
given in assertion i) of Theorem 3.1 in the Stanley's book \cite{Stan15}.
Namely, we have
\begin{align*}
C_{n}  &  =2\binom{2n}{n}-\frac{1}{2}\binom{2(n+1)}{n+1}=\binom{2n}{n}%
+\binom{2n}{n}-\frac{1}{2}\binom{2(n+1)}{n+1}\\
&  =\binom{2n}{n}+\frac{(2n)!}{n!n!}-\frac{(2n+1)!}{n!(n+1)!}=\binom{2n}%
{n}+\binom{2n}{n}(1-\frac{2n+1}{n+1})\\
&  =\binom{2n}{n}-\binom{2n}{n-1}.
\end{align*}
Assertion ii) presents, in fact, two cases of the so-called Touchard's
identity for odd and even cases.
\end{remark}

\begin{theorem}
\label{symet}We have:

i) for all $n\geq0$
\begin{equation}
\sum_{i=0}^{\left\lfloor n/2\right\rfloor }\binom{n}{2i}2^{n-2i}\binom{2i}%
{i}=\binom{2n}{n}, \label{idbin}%
\end{equation}

ii) for all $n\geq1$
\[
\binom{2n}{n}=2\times4^{n-1}-2^{n-1}\sum_{j=1}^{\left\lfloor n/2\right\rfloor
}\binom{n}{2j}\sum_{s=0}^{j-1}\frac{C_{s}}{4^{s}}.
\]

iii)
\[
C_{n+1}=2\binom{2n}{n}-\frac{1}{2}\sum_{j=0}^{\left\lfloor n/2\right\rfloor
}\binom{n}{2j}2^{n-2j}\binom{2j+2}{j+1}.
\]

iv)
\[
C_{n}\allowbreak=\allowbreak\sum_{k=0}^{n}(-1)^{k}\binom{n}{k}\binom
{k}{\left\lfloor k/2\right\rfloor }2^{n-k}.
\]

v)
\[
\frac{1}{2^{n-1}}\binom{n}{\left\lfloor n/2\right\rfloor }=2-\sum
_{j=0}^{\left\lfloor (n-1)/2\right\rfloor }\frac{1}{4^{j}}C_{j}.
\]

vi)
\[
S_{n}(3/2,3/2)\allowbreak=\allowbreak2\binom{n}{\left\lfloor n/2\right\rfloor
}\allowbreak-\binom{n+1}{\left\lfloor (n+1)/2\right\rfloor }.
\]

vii) Let us define $C_{-1}\allowbreak=\allowbreak-1/2$ and $d_{0}%
\allowbreak=\allowbreak1$ and further let $n\geq0:$
\[
d_{n}\allowbreak=\allowbreak\frac{n!}{2\times4^{n-1}}\sum_{k=0}^{n}%
(-1)^{k-1}\binom{2k}{k}C_{n-k-1}.
\]
Then%
\[
S_{n}(1,2)=\frac{\pi}{4}\sum_{k\geq0}\frac{S_{n+k}(3/2,3/2)}{2^{k}k!}d_{k}.
\]
In particular $\forall j\geq0$
\begin{align*}
\frac{-2\times2^{2j}}{2j+3}  &  =\frac{\pi}{4}\sum_{k\geq0}\frac{C_{j+k+1}%
}{2^{2k+1}(2k+1)!}d_{2k+1},\\
\frac{4^{j}}{2j+1}  &  =\frac{\pi}{4}\sum_{k\geq0}\frac{C_{j+k}}{4^{k}%
(2k)!}d_{2k}.
\end{align*}

\end{theorem}

\begin{proof}
i) We use Proposition \ref{Cent}a) and (\ref{iMm}). ii) We take $\gamma
\allowbreak=\allowbreak\delta\allowbreak=\allowbreak3/2$ and $\alpha
\allowbreak=\allowbreak\beta\allowbreak=\allowbreak1/2$. To use identity
(\ref{exp3}) we have to calculate only the coefficients $c_{n}(\alpha
-\gamma,\beta-\delta)\allowbreak=\allowbreak c_{n}(-1,-1).$ One can easily
check that
\[
c_{n}(-1,-1)=\left\{
\begin{array}
[c]{ccc}%
0 & \text{if} & n\text{ is odd}\\
n! & \text{if} & n\text{ is even}%
\end{array}
\right.  .
\]
Then we argue:%
\begin{align*}
\binom{2n}{n}  &  =\frac{B(3/2,3/2)}{4^{2(1/2-3/2)}B(1/2,1/2)}\sum
_{j=0}^{\left\lfloor n/2\right\rfloor }\binom{n}{2j}2^{n-2j}\sum_{s=0}%
^{\infty}\frac{C_{s+j}}{4^{s}}\\
&  =\frac{1}{2}\sum_{j=0}^{\left\lfloor n/2\right\rfloor }\binom{n}%
{2j}2^{n-2j}4^{j}\sum_{s=0}^{\infty}\frac{C_{s+j}}{4^{s+j}}=2^{n-1}\sum
_{j=0}^{\left\lfloor n/2\right\rfloor }\binom{n}{2j}(2-\sum_{s=0}^{j-1}%
\frac{C_{s}}{4^{s}}).
\end{align*}
Further we make use of (\ref{2}) and the following identity:
\[
\sum_{j=0}^{\left\lfloor n/2\right\rfloor }\binom{n}{2j}=2^{n-1}.
\]
iii) We take $\gamma\allowbreak=\allowbreak\delta\allowbreak=\allowbreak1/2$
and $\alpha\allowbreak=\allowbreak\beta\allowbreak=\allowbreak3/2$. We will
use Lemma \ref{pomoc}iiB with $i\allowbreak=\allowbreak1,$ then (\ref{iMm})
obtaining
\begin{align*}
C_{n+1}\allowbreak &  =\allowbreak\sum_{j=0}^{n}\binom{n}{j}2^{n-j}%
S_{j}(3/2,3/2)\\
&  =\sum_{j=0}^{\left\lfloor n/2\right\rfloor }\binom{n}{2j}2^{n-2j}C_{j}.
\end{align*}
Then we apply (\ref{Cnabin}) and finally we use (\ref{idbin}).

iv) let us take $\gamma\allowbreak=\allowbreak\delta\allowbreak=\allowbreak
3/2$ then, using (\ref{exp3}) and (\ref{2}) and the following argument
\begin{gather*}
(-1)^{n}\binom{n}{\left\lfloor n/2\right\rfloor }=\int_{-2}^{2}x^{n}%
g(x;-2,1/2,3/2)dx\\
=\int_{-2}^{2}x^{n}g(x;-2,3/2,3/2)\frac{1}{2+x}dx=\frac{1}{2}2^{n}(-1)^{n}%
\sum_{k=0}^{\infty}(-1)^{k+n}2^{-k-n}S_{n+k}(3/2,3/2)\\
=(-1)^{n}2^{n-1}(\sum_{m=0}^{\infty}(-1)^{m}2^{-m}S_{m}(3/2,3/2)-\sum
_{m=0}^{n-1}(-1)^{m}\frac{1}{2^{m}}S_{m}(3/2,3/2).
\end{gather*}
Consequently we have%
\[
\frac{1}{2^{n-1}}\binom{n}{\left\lfloor n/2\right\rfloor }=2-\sum
_{j=0}^{\left\lfloor (n-1)/2\right\rfloor }\frac{1}{4^{j}}C_{j}.
\]

v) We take $\alpha\allowbreak=\allowbreak\beta\allowbreak=\allowbreak3/2$ and
$\gamma\allowbreak=\allowbreak3/2$ and $\delta\allowbreak=\allowbreak1/2.$

vi) We take $\allowbreak\delta\allowbreak=\allowbreak3/2$ and $\gamma
\allowbreak=\allowbreak3/2$ and $\alpha\allowbreak=\allowbreak1$ and
$\beta\allowbreak=2.$ then we have%
\begin{align*}
S_{n}(1,2)  &  =\left\{
\begin{array}
[c]{ccc}%
\frac{2^{n}}{n+1} & \text{if} & n\text{ is even}\\
-\frac{2^{n}}{n+2} & \text{if} & n\text{ is odd}%
\end{array}
\right. \\
&  =\frac{1}{8}\int_{-2}^{2}y^{n}(2-y)dy.
\end{align*}
We will be applying (\ref{exp3}), hence we have to calculate coefficients
$c_{k}(\alpha-\gamma,\beta-\delta)\allowbreak=\allowbreak c_{k}(-1/2,1/2).$
Now
\begin{align}
c_{s}(-1/2,1/2)\allowbreak &  =\allowbreak\sum_{k=0}^{s}\binom{s}{k}%
(-1)^{s-k}(-1/2)_{\left(  k\right)  }(1/2)_{\left(  s-k\right)  },\nonumber\\
&  =\sum_{k=0}^{s}\binom{s}{k}(-1)^{s-k}(-1)^{k}(1/2)^{\left(  k\right)
}(1/2)_{\left(  s-k\right)  }\nonumber\\
&  \sum_{k=0}^{s}\binom{s}{k}(-1)^{k}(1/2)^{\left(  k\right)  }(-1/2)^{\left(
s-k\right)  } \label{ckk}%
\end{align}
Now notice that if $k\allowbreak=\allowbreak0$ then $(1/2)^{k}\allowbreak
=\allowbreak(1/2)_{k}\allowbreak=\allowbreak1$ while for $k>0$ we get
\begin{align*}
(1/2)^{(k)}  &  =\prod_{m=0}^{k-1}(1/2+m)\allowbreak=\allowbreak\frac
{(2k)!}{4^{k}k!},\\
(-1/2)^{(k)}  &  =(-1)^{k}\prod_{m=0}^{k-1}(m-1/2)=-\frac{(2(k-1))!}%
{2\times4^{k-1}(k-1)!}.
\end{align*}
Hence
\begin{align*}
c_{n}(-1/2,1/2)  &  =(-1)^{n}\frac{(2n)!}{4^{n}n!}-\sum_{k=0}^{n-1}%
(-1)^{k}\frac{n!(2k)!(2(n-k-1))!}{k!(n-k)!4^{k}k!2\times4^{n-k-1}(n-k-1)!}\\
&  (-1)^{n}\frac{(2n)!}{4^{n}n!}-\frac{n!}{2\times4^{n-1}}\sum_{k=0}%
^{n-1}(-1)^{k}\binom{2k}{k}C_{n-k-1}.
\end{align*}
Notice, that if we define $C_{-1}$ as $-1/2$ then the formula above, becomes
more simple and has the more unified form:%
\[
c_{n}(-1/2,1/2)=\frac{n!}{2\times4^{n-1}}\sum_{k=0}^{n}(-1)^{k+1}\binom{2k}%
{k}C_{n-k-1}.
\]
Upon applying ((\ref{exp3}), we get:
\[
S_{n}(1,2)=\frac{\pi}{4}\sum_{k=0}^{\infty}\frac{S_{n+k}(3/2,3/2)}{2^{k}%
k!}c_{k}(-1/2,1/2).
\]
i.e., assertion vii).
\end{proof}

\begin{remark}
Notice, that assertion v) of Theorem \ref{symet} is a generalization of the
assertion v) of the Theorem \ref{c=0}.
\end{remark}

\begin{remark}
Let us notice, that the sequence $\left\{  d_{n}\right\}  _{n\geq0}$ defined
in assertion viii) of Theorem \ref{symet} is also a moment sequence since it
is originally defined by the formula (\ref{ckk}). This is so since $\left\{
\left(  -1/2\right)  ^{(n)}/n!\right\}  $ and $\left\{  \left(  1/2\right)
^{(n)}/n!\right\}  $ are both moment sequences by (\ref{mJ}), $\left\{
n!\right\}  $ is the is the moment sequence of the distribution with the
density $\exp(-x)$ on $\mathbb{R}^{+}$. \ Consequently, by Proposition
\ref{momenty} we deduce that $\left\{  d_{n}\right\}  $ is a moment sequence.
Further we have for $\left\vert x\right\vert <1$;%
\begin{gather*}
\sum_{n\geq0}\frac{x^{n}}{n!}d_{n}\allowbreak=\allowbreak\sum_{n\geq0}%
\sum_{k=0}^{n}\frac{x^{k}}{k!}(-1)^{k}(1/2)^{(k)}\frac{x^{n-k}}{(n-k)!}%
(-1/2)^{(n-k)}\allowbreak\\
=\allowbreak\sum_{k=0}^{\infty}\frac{x^{k}}{k!}(-1)^{k}(1/2)^{(k)}\sum_{n\geq
k}\frac{x^{n-k}}{(n-k)!}(-1/2)^{(n-k)}=\sqrt{\frac{1-x}{1+x}}.
\end{gather*}
But $\sqrt{\frac{1-x}{1+x}}\allowbreak=\allowbreak\exp(\tanh^{-1}(-x))$ for
$x$ real and $\left\vert x\right\vert <1$. Hence we can identify sequence
$\left\{  d_{n}\right\}  $ as $sc$ version of the sequence A000246 in OEIS.

It might be of interest to notice that the formula given in assertion viii) of
Theorem \ref{symet} expresses $d_{n}$ in terms of the Catalan numbers.
\end{remark}


\begin{thebibliography}{99}                                                                                               %


\bibitem {Carr19}J. Alonso-Carre\'{o}n, J. L\'{o}pez-Bonilla, J and G. Thapa,
(2020). \emph{A Note on a Formula of Riordan Involving Harmonic Numbers.}
Journal of the Institute of Engineering. \textbf{15}(2020). 10.3126/jie.v15i1.27738.

\bibitem {Berry11}P. Barry, \emph{Riordan arrays, orthogonal polynomials as
moments, and Hankel transforms.} J. Integer Seq. \textbf{14} (2011), no.
\textbf{2}, Article 11.2.2, 37 pp. MR2783956

\bibitem {Benn11}G. Bennett, \emph{Hausdorff means and moment sequences}.
Positivity \textbf{15} (2011), no. 1, 17--48. MR2782745

\bibitem {Bern99}F.R. Bernhart, \emph{Catalan, Motzkin, and Riordan numbers}.
Discrete Math\emph{.} \textbf{204} (1999), no. 1-3, 73--112. MR1691863

\bibitem {Dona77}P. Donaghey, Robert, L. W. Shapiro, \emph{Motzkin numbers.}
J. Combinatorial Theory Ser. A \textbf{23} (1977), no. 3, 291--301. MR0505544

\bibitem {Grimaldi12}R.P. Grimaldi, \emph{Fibonacci and Catalan numbers. An
introduction.} John Wiley \& Sons, Inc., Hoboken,\emph{ }NJ, 2012. xiv+366 pp.
ISBN: 978-0-470-63157-7 MR2963306

\bibitem {Rioe66}J. Riordan, \emph{Combinatorial identities}. John Wiley \&
Sons, Inc\emph{.}, New York-London-Sydney 1968, xiii+ 256 pp. MR0231725

\bibitem {Roman15}S. Roman, \emph{An introduction to Catalan numbers.} With a
foreword by Richard Stanley. Compact Textbooks in Mathematics.
Birkh\"{a}user/Springer, Cham, 2015. xii+121 pp. ISBN: 978-3-319-22143-4;
978-3-319-22144-1 MR3380815

\bibitem {Sloan}N.J.A. Sloane, \emph{The on-line Encyclopedia of Integer
Sequences}, \texttt{https://oeis.org}.

\bibitem {Sokal20}A. D. Sokal, \emph{The Euler and Springer numbers as moment
sequences}. Expo. Math. \textbf{38} (2020), no. 1, 1--26. MR4082303

\bibitem {Stan15}R. P. Stanley, \emph{Catalan numbers.} Cambridge University
Press, New York, 2015. viii+215 pp. ISBN: 978-1-107-42774-7; 978-1-107-07509-2 MR3467982

\bibitem {Szablowski2010(1)}P. J. Szab\l owski, \emph{Expansions of one
density via polynomials orthogonal with respect to the other.} J. Math. Anal.
Appl. \textbf{383} (2011), no. 1, 35--54. MR2812716,
\texttt{http://arxiv.org/abs/1011.1492}

\bibitem {SzabChol}P.J. Szab\l owski, \emph{A few remarks on orthogonal
polynomials}, Appl. Math. Comput. \textbf{252} (2015), 215--228.
\texttt{http://arxiv.org/abs/1207.1172}

\bibitem {Szabl21}P.J. Szab\l owski, \emph{On positivity of orthogonal series
and its applications in probability}, Positivity\emph{, }\textbf{26}:19(2022),
on-line, \texttt{https://arxiv.org/abs/2011.02710}.

\bibitem {Szab23}P.J. Szab\l owski, \emph{Moment sequences and differential
equations}, \texttt{https://arxiv.org/abs/2204.04706}, submitted
\end{thebibliography}
\end{document}